\documentclass[12pt, reqno, a4paper]{amsart}

\usepackage{amssymb}
\usepackage{enumitem}
\usepackage{bm}
\usepackage[left=3.3cm, right=3.3cm, top=3.5cm, bottom=3.5cm]{geometry}
\usepackage{microtype}

\theoremstyle{plain} 
\newtheorem{thm}{Theorem}

\newtheorem{lem}[thm]{Lemma}

\newtheorem{prop}[thm]{Proposition}

\newtheorem{obs}[thm]{Observation}
\newtheorem{conj}[thm]{Conjecture}
\newtheorem*{cl-non-num}{Claim}

\newcommand{\whp}{{\rm whp}}
\newcommand{\C}{\mathcal{C}}

\newcommand{\bC}{\tilde{C}}
\newcommand{\bA}{\tilde{A}}

\DeclareMathOperator{\comp}{comp}

\begin{document}

\title{Linear Extensions and Comparable Pairs in Partial Orders}

\author{Colin McDiarmid}
\address{Department of Statistics, University of Oxford, 
Oxford, UK}
\email{cmcd@stats.ox.ac.uk}

\author{David Penman}
\address{Department of Mathematical Sciences, University of Essex, 
Colchester, UK}
\email{dbpenman@essex.ac.uk}

\author{Vasileios Iliopoulos}
\address{Department of Mathematical Sciences, University of Essex, 
Colchester, UK}
\email{viliop@essex.ac.uk, iliopou@gmail.com}



\date{\today}

\begin{abstract}
We study the number of linear extensions of a partial order 
with a given proportion of comparable pairs of elements, and estimate the maximum and minimum possible numbers.
We also consider a random interval partial order on $n$ elements, which has close to 
a third of the pairs comparable with high probability: we show that the number of 
linear extensions is $n! \, 2^{-\Theta(n)}$ with high probability.
\end{abstract}

\maketitle

\section{Introduction} 
\label{sec.intro}

Initially our interest was in the random interval order $P_n$, where we pick $n$ 
intervals independently and uniformly at random in $(0,1)$ (see the 
start of Section~\ref{sec.randint} for a precise definition).
It turns out that with high probability about a third of the $\binom{n}{2}$ 
possible pairs are comparable, and the number $e(P_n)$ of linear extensions 
is $n! \, 2^{-\Theta(n)}$.  But, given the proportion of comparable 
pairs, is this a large number of linear extensions?  
How few or many could there be? The main part of the paper is devoted to answering such questions.  

We investigate how few or many linear extensions a general partially 
ordered set $Q$ or $Q_n$ of $n$ points may have, when it has a 
certain proportion of the possible edges in its comparability graph 
(we recall basic definitions in the next section).
It is well known -- see for example the Proposition at the 
end of~\cite{ehs}, or see~\cite{Tro} -- that the comparability graph determines 
the number of linear extensions $e(Q)$, in the sense that posets
with isomorphic (undirected) comparability graphs have the same number of linear extensions. 

This number $e(Q)$ is an important invariant of a poset which is related 
to, for example, the worst case number $c(Q)$ of pairwise comparisons 
required to determine a (hidden) total order when we are initially given
a partial order compatible with it. It is easy to see by a bisection 
argument that $c(Q) \geq \log_{2} e(Q)$; and there has been much work 
over the years on how close to the truth this is, with one well known result 
being the bound of Fredman~\cite{fred} that $c(Q) \leq \log_{2} e(Q)+2n$, see 
for example~\cite{cardinal} for a recent algorithmic result 
in this direction and references to earlier literature. 

We focus primarily on posets $Q_n$ of $n$ points where the comparability 
graph is dense, with about $\delta \binom{n}{2}$ edges for some 
constant $\delta \in (0, 1)$ -- we refer to such a poset as a dense poset -- but we 
are interested also in what happens when $\delta=\delta(n)$ tends to 0 or 1 as $n \to \infty$.
Given a positive integer $n$ and $0 < \delta < 1$, we let
\[ f^+(n,\delta)= \max \left\{ e(Q_n): \comp(Q_n) \geq \delta \binom{n}{2} \right\} \]
and
\[ f^-(n,\delta)= \min \left\{ e(Q_n): \comp(Q_n) \leq \delta \binom{n}{2} \right\}, \]
where $\comp(Q_n)$ denotes the number of edges in the comparability graph of the poset $Q_n$.
Of course $1 \leq f^-(n,\delta), f^+(n,\delta) \leq n!$.

We are interested in the asymptotic values of functions $f^+(n,\delta)$ and $f^-(n,\delta)$. 
Some steps in the investigation of the precise maximum values $f^+(n,\delta)$ 
were taken in~\cite{fishb}, which concentrated on small numbers of edges  
rather than the comparatively large values we shall mainly examine. 
It is known from~\cite{fishb} that a partial order with the maximum number of 
linear extensions for a given number of vertices and comparable pairs is a semiorder.  
We are not aware of any previous work on the minimum values $f^-(n,\delta)$. 

The main results of our paper are outlined in Theorem~\ref{thm.main} below. 
The main thrust is that, for each fixed $0<\delta<1$, we have $f^+(n,\delta) = n! \, 2^{-\Theta(n)}$ 
and $f^-(n,\delta) = 2^{\Theta(n)}$. Let us spell this out more fully.

\begin{thm}
\label{thm.main}
The maximum values $f^+(n,\delta)$ and minimum values $f^-(n,\delta)$ satisfy the following, as $n \to \infty$.
\begin{enumerate}
\item[(a)]
For each $0<\delta<1$ there are constants $c_1(\delta)$ and $c_2(\delta)$  such that
\[0< c_1(\delta) <  \left( \frac{f^+(n,\delta)}{n!} \right)^{1/n} <  c_2(\delta) <1  
\;\;\; \mbox{ for } n \mbox{ sufficiently large};\]
and further, $c_1(\delta) \to 1$ as $\delta \to 0$, and $c_2(\delta) \to 0$ as $\delta \to 1$.

\item[(b)]
For each $0<\delta<1$ there are constants $c_3(\delta)$ and $c_4(\delta)$ such that 
\[1 < c_3(\delta) <  f^-(n,\delta)^{1/n} <  c_4(\delta) < \infty \;\;\; \mbox{ for } n \mbox{ sufficiently large};\]
and further,  $c_3(\delta) \to \infty$ as $\delta \to 0$, and $c_4(\delta) \to 1$ as $\delta \to 1$.
\end{enumerate}
\end{thm}

\subsection*{Plan of the paper}

After some preliminaries in the next section, in Section~\ref{models} we study 
posets consisting of either disjoint chains or disjoint antichains, and derive 
bounds on the numbers of linear extensions and comparable pairs.

In Section~\ref{sec.max} we consider the maximum values $f^+(n,\delta)$.  
We prove part (a) of Theorem~\ref{thm.main}, and give values for $c_1(\delta)$ and $c_2(\delta)$, see~\eqref{eqn.c1def} and~\eqref{eqn.c2def}.  Also, in~\eqref{eqn.vdense} we give a formula for $f^+(n,\delta)$ in the very dense case, and in Proposition~\ref{prop.plusrate} we describe the rate of convergence of $(f^+(n,\delta)/n!)^{1/n}$ 
to $0$ as $\delta=\delta(n) \to 1$.
 
In Section~\ref{sec.min} we similarly consider the minimum values $f^-(n,\delta)$.
We prove part (b) of Theorem~\ref{thm.main}, and give values for $c_3(\delta)$ and $c_4(\delta)$, see~\eqref{eqn.c3def} 
and~\eqref{eqn.c4def}. Also, in~\eqref{eqn.vsparse2} we give a formula for $f^-(n,\delta)$ 
in the very sparse case, and in Proposition~\ref{prop.minusrate} we describe the rate of growth of $f^-(n,\delta)^{1/n}$ as $\delta=\delta(n) \to 0$.

In Section~\ref{sec.randint} we investigate the random interval order $P_n$ (getting the historical 
sequence of events out of order): we show that with high probability $\comp(P_n)$ is about 
$\frac13 \binom{n}{2}$ and $e(P_n)$ is $n! \, 2^{-\Omega(n)}$, 
and thus \whp\ $e(P_n)= n! \, 2^{-\Theta(n)}$ by part (a) of Theorem~\ref{thm.main}. 
Finally, in Section~\ref{sec.concl}, we make a few concluding remarks and conjectures; and in particular we conjecture that
$(f^+(n,\tfrac12)/n!)^{1/n} \to \tfrac12$ and $f^-(n,\tfrac12)^{1/n} \to 2$, as $n \to \infty$.

\section{Preliminaries}
\label{sec.prelims}

A \emph{poset} (partially ordered set) $P$ is a set of points equipped 
with an irreflexive, antisymmetric, and transitive relation $\prec \,$, 
see e.g.~\cite{bri},~\cite{stan}. All posets in this paper are finite. 
Typical notations for a poset will be $P$ or $Q_n$, where $n$ indicates the number of points. 
A \emph{linear extension} of a poset is a total order $<$ on the 
ground set of points such that whenever $x \prec y$ in the partial order, 
then we have $x < y$ in the total order too. The number of linear extensions of a poset $P$ is denoted by $e(P)$.

We say that two points $x \neq y$ are \emph{comparable} in a poset if 
$x \prec y$ or $y \prec x$; otherwise they are \emph{incomparable}.
The \emph{comparability graph} $G$ is the undirected graph with 
vertex set the set of points and an edge between $x$ and $y$ 
if and only if they are comparable: the \emph{incomparability graph} 
is the complement $\overline{G}$ of the comparability graph $G$.  

A {\em chain} in $P$ is a set $T$ of points, any two of which are 
comparable: such a set $T$ can be totally ordered, i.e. enumerated as 
$T=\{t_{1}, t_{2}, \ldots, t_{r}\}$ with $t_{1} \prec t_{2} \prec \ldots \prec t_{r}$. 
The maximum number of points in a chain is called the {\em height} of $P$ 
and is denoted by $h(P)$. By contrast, an {\em antichain} in $P$ is a set of 
points, no two of which are comparable. The maximum number of points in an 
antichain is called the {\em width} of $P$ and is denoted by $w(P)$. 

Recall that the comparability graph of $P$ determines $e(P)$. 
If $P$ and $Q$ are posets on the same set of points,
and the comparability graph of $P$ is a strict subgraph of that of $Q$, 
then clearly $e(P)>e(Q)$ -- see e.g.~\cite{stachowiak}.   
However, the na\"{i}ve intuition that the more edges in the 
comparability graph the fewer linear extensions is inaccurate: for example 
the six-vertex poset consisting of two three-element chains with no comparabilities 
between them has $6$ edges in the comparability graph and $\binom{6}{3}=20$ 
linear extensions (as we shall see), but the six-element poset comprising of 
one maximum element, one minimum element and an antichain of four elements 
all between the maximum and minimum element has $9$ edges in the comparability 
graph and $4!=24$ linear extensions. We thus word our results in terms of 
posets which have at most, or at least, a certain proportion of edges in 
the comparability graph to mitigate these lack of monotonicity issues. 

We shall sometimes use the \emph{level structure} of the poset $P$ on a 
set $V$ of points. The first level $L_{1}$ consists of all 
minimal elements, that is, those points $x\in V$ for 
which $y \preceq x \Rightarrow y=x$. The second level consists 
of the minimal elements of the induced subposet on $V \setminus L_{1}$ 
and generally the $i$th level $L_{i}$ consists of the minimal elements 
of the induced subposet on $V \setminus \cup_{j=1}^{i-1}L_{j}$. 
Note that the number of non-empty levels is $h(P) \leq \vert V \vert$ and that each level is an antichain. 

For positive functions $f(n)$ and $g(n)$ we as usual write $f(n)=O(g(n))$ if there is a constant $K$ such that $f(n) \leq K \cdot g(n)$ for all 
large enough $n$, and similarly $f(n)=\Omega(g(n))$ if $f(n) \geq K \cdot g(n)$ 
for all large enough $n$. If $f(n)=O(g(n))$ and $f(n)=\Omega(g(n))$, then we write 
$f(n)=\Theta(g(n))$. We also use the standard notations $f(n)=o(g(n))$ if $\lim_{n \to \infty} \frac{f(n)}{g(n)}=0$, 
and $f(n) \sim g(n)$ if $f(n)=(1+o(1))g(n)$ as $n \to \infty$. It is a little less standard
to use the notation $f(n)=\omega(g(n))$ when $g(n)=o(f(n))$, that is, when 
$\lim_{n \to \infty} \frac{f(n)}{g(n)}=\infty$. 

When we talk later about properties of a random poset $Q_n$ with $n$ points, we shall say that $Q_n$  
has a given property $\wp$ \emph{with high probability} ({\whp}) if 
\[ 
\mathbb{P}(Q_n \mbox{ has } \wp ) \to 1 \;\; \mbox{ as } \; n \to \infty.
\]
We use Stirling's formula $n! \sim \sqrt{2\pi n}(n/e)^{n}$, and make repeated    
use of the related inequality $n! \geq (n/e)^{n}$ which follows from the 
series expansion of $e^{n}$. Also, we write $[n]$ to denote the set of positive integers $\{1, \ldots, n\}$.

\subsection*{Graphs and entropy}

A subset of the vertices of a graph $G$, such that every two are connected by an edge, is a called a \emph{clique}. 
The \emph{clique number} of $G$ is the cardinality of the largest clique and
the \emph{chromatic number} $\chi(G)$ of $G$ is the minimum number of colours needed 
to colour its vertices in such a way that any two vertices joined by an edge 
receive different colours. $G$ is \emph{perfect} if  
for every induced subgraph of $G$ the chromatic number equals 
the clique number, see for example~\cite{RA-Reed}.  
It is well known that the comparability graph of a poset is perfect, and the complement of a perfect graph is perfect.

Given a graph $G$ on vertex set $[n]$, the \emph{clique polytope} $\C(G)$ of $G$ is defined by
\[
\C(G) = \left \{ \bm {x} \in [0, 1]^{n}: \sum_{i \in C} x_i \leq 1 \text{~for~each~clique~} C \text{~of~} G \right \},
\]
see for example~\cite{Chvatal}. 
Stanley~\cite{Stanley} proved that, if $G$ is the comparability graph of a poset $Q_n$, then 
the number of linear extensions $e(Q_n)$ is equal to the volume of $\C(G)$ multiplied by $n!$.  

The \emph{entropy} $H(G)$ of a graph $G$ is defined by      
\begin{equation*}
H(G)=\min_{\bm{x} \in \C(G)} - \dfrac{1}{n} \sum_{i=1}^{n} \log_2 x_i.   
\end{equation*}
It satisfies $0 \leq H(G) \leq \log_2 n$, and 
indeed $H(G) \leq \log_2 \chi(G)$ (see e.g.~\cite{Simo}). 
Let $Q_n$ be a partial order on $[n]$, with comparability graph $G$ and incomparability graph $\overline{G}$.  
Then, by~\cite[Theorem 2]{csi},       
\[
H(G) +  H(\overline{G}) = \log_2 n.
\]
We are interested here in graph entropy since there are bounds on the number 
of linear extensions of $Q_n$ in terms of the entropy of $G$ or $\overline{G}$.
Kahn and Kim~\cite{kk} proved
\begin{equation} \label{eqn.KK}
n! \, 2^{-n H(G)} \leq e(Q_n) \leq 2^{n H(\overline{G})},
\end{equation}
and Cardinal et al.~\cite{cardinal} proved 
\begin{equation} \label{eqn.Cardinal}
 e(Q_n) \geq 2^{\frac{1}{2} n H(\overline{G})}.
\end{equation}

\section{The chain and antichain examples}
\label{models}

We introduce two important standard examples of partial orders where we can write down the number of linear extensions.
The chain example $C=C(n_1,\ldots,n_k)$ is the poset consisting of 
disjoint chains of sizes $n_1, \ldots, n_k$ with no comparabilities between the chains.  We have
\begin{equation} \label{eqn.chain1a}
e(C) = \binom{n}{n_1, \ldots, n_k} \leq k^n.
\end{equation} 
To see the equality here, when for example $k=2$, observe that the 
only choices to make are which positions in the total order are occupied by the $n_1$ points from the first chain. 
It follows that for any poset $Q_n$ with width at most $k$ we have
\begin{equation} \label{eqn.chain1b}
e(Q_n)  \leq k^n,
\end{equation} 
since by Dilworth's Theorem~\cite{dilworth} a poset $Q_n$ with $w(Q_n) \leq k$ can 
be partitioned into $k$ or fewer chains.

The antichain example is the poset $A=A(n_1,\ldots,n_k)$ consisting of 
disjoint antichains $A_1,\ldots, A_k$ of sizes $n_1, \ldots, n_k$ such 
that if $x \in A_i$, $y \in A_j$ and $i<j$ then $x \prec y$.
Observe that the comparability graphs for the posets $A$ and $C$ 
are complementary (with the natural choice of sets of points).
We have
\begin{equation} \label{eqn.chain2a}
e(A) = \prod_{i=1}^{k} n_i! = n! \bigg /  \binom{n}{n_1, \ldots, n_k} \geq n! / k^n.
\end{equation}  
It follows by considering the level structure
that for any poset $Q_n$ with height at most $k$ we have
\begin{equation} \label{eqn.chain2b}
e(Q_n)  \geq n! / k^n.
\end{equation}

As an aside, recall that by the Kleitman-Rothschild Theorem~\cite{kr}, almost all 
posets $Q_n$ on $[n]$ have height $3$, and thus by~\eqref{eqn.chain2b} $e(Q_n) \geq n! \, 3^{-n}$ for almost all such posets.
In fact, by~\cite{Bri} almost all posets $Q_n$ have
\[ (e(Q_n)/n!)^{1/n} \sim 2^{-\tfrac32} \approx 0.35. \]  
See~\cite{BPS} and the references there for much more precise results 
on this and on the average number (which is a factor of order $\sqrt{n}$ larger than where $e(Q_n)$ is concentrated).

We shall be most interested in the special cases when the $n_i$ are as 
balanced as possible. The balanced chain example $\bC(n,k)$ is defined, 
for all integers $1 \leq k \leq n$, as $C(n_1, \ldots, n_k)$ where 
$\lfloor n/k \rfloor \leq n_1 \leq \cdots \leq n_k \leq \lceil n/k \rceil$, and the $n_i$ sum to $n$.
Then $e(\bC(n,k)) \leq k^n$ by~\eqref{eqn.chain1a} or~\eqref{eqn.chain1b}.
Also, for each given point there are at most $\lceil n/k \rceil -1 \leq \frac{n-1}{k}$ comparable points, so 
\begin{equation} 
\label{eqn.chaincomp}
\comp(\bC(n,k)) \leq \frac{1}{2} n \frac{n-1}{k} = \frac{1}{k} \binom{n}{2}.
\end{equation}
The balanced antichain example $\bA(n,k)$ is defined similarly, 
for all integers $1 \leq k \leq n$, as $A(n_1,\ldots,n_k)$, 
where the $n_i$ are as above. Then $e(\bA(n,k)) \geq n! / k^n$ by~\eqref{eqn.chain2a} or~\eqref{eqn.chain2b}.
Also, by~\eqref{eqn.chaincomp} and taking complements
\begin{equation} 
\label{eqn.antic2}
\comp(\bA(n,k)) \geq \left(1-\frac1{k}\right) \binom{n}{2}.
\end{equation}

For $1 \leq t \leq (n+1)/2$, we shall also use the examples $C(t,n-t)$ and $A(t,n-t)$
where the poset consists of two disjoint chains in the former 
case and two disjoint antichains in the latter, of cardinalities $t$ and $n-t$.  
We set $\tau := t/n.$ 
Then 
\begin{equation} \label{eqn.chain3}
e(C(t,n-t))= \binom{n}{t}  \leq \left(\frac{e}{\tau}\right)^{\tau n}, 
\end{equation}
and
\begin{equation} \label{eqn.chain4}
\comp(C(t,n-t)) = \binom{n}{2} - \tau(1-\tau) n^{2} 
\leq (1- \tau) \binom{n}{2},
\end{equation}
since $(1-\tau)n^2 \geq \tfrac{n-1}{2n} n^2 = \binom{n}{2}$.
Also
\begin{equation} \label{eqn.antic3}
e(A(t,n-t))= t! (n-t)! =n! \bigg / \binom{n}{t} \geq n! \left(\frac{\tau}{e} \right)^{\tau n},
\end{equation}  
and
\begin{equation} \label{eqn.antic4}
\comp(A(t,n-t)) = \tau(1-\tau) n^{2} \geq \tau \binom{n}{2}.
\end{equation}
Note that $f(\tau)=\left(\frac{\tau}{e} \right)^{\tau}$ is continuous and decreasing on $[0,1]$, and $f(0)=1$.  Thus
$g(\tau)=\left(\frac{e}{\tau} \right)^{\tau}=1/f(\tau)$ is continuous and increasing on $[0,1]$, with $g(0)=1$.

\section{Maximum numbers of extensions: $f^+(n,\delta)$}
\label{sec.max}

In this section we prove part (a) of Theorem~\ref{thm.main}, and give values 
for $c_1(\delta)$ and $c_2(\delta)$, see~\eqref{eqn.c1def} and~\eqref{eqn.c2def}.  
Also, in~\eqref{eqn.vdense} we give a formula for $f^+(n,\delta)$ in the very dense 
case, and in Proposition~\ref{prop.plusrate}, we describe the rate of 
convergence of $(f^+(n,\delta)/n!)^{1/n}$ 
to $0$ as $\delta=\delta(n) \to 1$.

\subsection*{Lower bounds}

Let $0 \leq \delta<1$. Let $k=\lceil (1-\delta)^{-1} \rceil$,
and let $Q_n$ be the balanced antichain example $\bA(n,k)$. 
Then by~\eqref{eqn.antic2}, since $1-\frac{1}{k} \geq \delta$, 
we have $\comp(Q_n) \geq \delta \binom{n}{2}$; 
and since $k \leq \tfrac{2-\delta}{1-\delta}$ we have
\[ e(Q_n) \geq n! \, k^{-n} \geq n! \, \left(\frac{1-\delta}{2 -\delta}\right)^n. \]
Thus
\begin{equation}  \label{eqn.fplusgeq1}
\left( \frac{f^+(n,\delta)}{n!} \right)^{1/n} \geq \left( \frac{e(Q_n)}{n!} \right)^{1/n} 
\geq \; \frac{1-\delta}{2-\delta} \; > \; \frac12 (1 - \delta).
\end{equation}
This result covers all $\delta$ with $0 \leq \delta<1$, but even for 
$\delta=0$ the lower bound is only $\tfrac12$, and we want a lower 
bound which approaches $1$ as $\delta \to 0$.
 
Assume that $0 < \delta \leq \tfrac12$. Let $t= \lceil \delta n \rceil$,      
let $\tau=t/n$, and let $Q_n$ be the two-antichain example $A(t,n-t)$.
By~\eqref{eqn.antic4} since $\delta \leq \tau < \delta + \tfrac1n$   
we have  
$\comp(Q_n) \geq \delta \binom{n}{2}$;    
and then by~\eqref{eqn.antic3}, since $(x/e)^x$ is decreasing on $[0, 1]$    
we have                     
\begin{equation}  \label{eqn.fplusgeq2}
\left( \frac{f^+(n,\delta)}{n!} \right)^{1/n} \geq \left( \frac{e(Q_n)}{n!} \right)^{1/n} \geq 
\left ( \frac{\delta + \tfrac1n}{e} \right )^{\delta + \frac1n}
\geq \left( \frac{2 \delta}{e} \right)^{2 \delta} 
\end{equation}
for  
$\tfrac1n \leq \delta \leq \tfrac12$.              
By~\eqref{eqn.fplusgeq1} and~\eqref{eqn.fplusgeq2}, we may set  
\begin{equation} \label{eqn.c1def}
c_1(\delta) = 
\begin{cases}
  \max \left \{ \frac{1-\delta}{2-\delta} \, , \: \left( \frac{2 \delta}{e} \right)^{2 \delta} \right \}  
    & \text{ if \,\,\,} 0 < \delta \leq \frac12 \cr
  \frac{1-\delta}{2-\delta} \hspace{1.17in}  
    & \text{ if \,\,\,} \frac12 < \delta <1.
\end{cases}
\end{equation}
Note that the lower bound $c_1(\delta)$ tends to 1 as $\delta \to 0$, as desired.

\subsection*{Upper bounds}

We first use a martingale concentration inequality to prove the upper 
bound $e(Q_n) \leq n! \, 2^{-\Omega(n)}$ on the number of linear extensions of a dense poset.  

\begin{lem}
\label{lem.fplus-ub1}
There is an absolute positive constant $c$, which can be taken to be $\frac{\log_2 e}{32} \approx 0.045$, 
such that the following holds. 
Let $0 < \delta \leq 1$, and let $Q_n$ be a poset on $n$ points with 
$\comp(Q_n) \geq \delta \binom{n}{2}$.
Then 
\[   
e(Q_n) \leq  n! \, 2^{- c \delta^2 n}. 
\]
\end{lem}   
\begin{proof}
Let $S_n$ be the set of all $n!$ permutations of $[n]$.
For $\tau \in S_n$ let $g(\tau)$ be the number of conflicts of $\tau$ with 
the poset $Q_n$: that is, the number of ordered pairs 
$(i, j)$, such that $i < j$ in $\tau$ and $j \prec i$ in the partial order. 
Obviously, $\tau$ is a linear extension of $Q_n$ if and only if $g(\tau)=0$.
If $\tau$ and $\tau'$ differ by one transposition $\sigma$, i.e. $\tau=\tau^{\prime}\circ \sigma$ 
in the symmetric group ($\sigma$ acting first say), 
then if $\sigma$ interchanges $x$ and $y$, it changes the number of conflicts 
by at most $n-1$ for each of $x$ and $y$, so 
\begin{equation}
\label{bounded}
\vert g(\tau')-g(\tau)\vert \leq 2(n-1).
\end{equation} 
When $\tau$ is chosen uniformly at random from $S_n$, 
we have that 
\[ \mathbb{E}(g(\tau)) = \comp(Q_n)/2 \]
since for each $x \neq y$,
$\mathbb{P}(x < y) = \mathbb{P}(y < x) = \frac{1}{2}$ (where $<$ refers to the ordering $\tau$).
Put $s=\comp(Q_n)/2 \geq \delta n(n-1)/4$. 
We may now use Theorem $7.4$ of~\cite{McD}, together with the inequality~\eqref{bounded}, to obtain
\begin{eqnarray*}
\mathbb{P}(g(\tau)=0) 
& = &
\mathbb{P}(g(\tau)- s \leq -s)  \\
& \leq &
\exp \left(\frac{-2s^2}{4n(n-1)^2} \right) \; \leq \; e^{- \delta^2 n/32}.
\end{eqnarray*}
Hence
\[ e(Q_n) = n! \, \mathbb{P}(g(\tau)=0) \leq n! \, e^{- \delta^2 n/32} \]
which completes the proof. 
\end{proof}

The last lemma covers all $\delta$ with $0<\delta \leq 1$, but 
even when $\delta=1$ we just find  $e(Q_n) \leq  n! \, 2^{- c n}$. 
We need another result to show that when $\delta \to 1$ we have that
$(f^+(n,\delta)/n!)^{1/n}=o(1)$.
This follows from Lemma~\ref{lem.maxub} below (which says nothing 
asymptotically if $\delta< 1- 2/e \approx 0.264$). 
In fact we shall see that if $\delta \to 1$, but not too quickly, then
$(f^+(n,\delta)/n!)^{1/n} = \Theta(1\!-\!\delta)$. 
First, here is a preliminary observation: we omit the easy proof.   
\begin{obs} 
\label{obs.nearchain}
Let $Q$ be a partial order on a set $S$ of at least two elements, 
and let $v \in S$. Suppose that $S'=S \setminus \{v\}$ is a chain, 
and $v$ is incomparable to exactly $x$ points in $S'$. Then $e(Q)=x+1$.
\end{obs}

\begin{lem}
\label{lem.maxub}
Let $0<\delta<1$ and suppose that $\comp(Q_n) \geq \delta \binom{n}{2}$.  Then
\[ e(Q_n) \leq n! \: e^{\frac{2}{1-\delta}}  \left(\frac{e(1-\delta)}{2} \right)^n. \]
\end{lem}
\begin{proof}
Let $\overline{G}$ be the complement of the comparability graph $G$ of the 
poset $Q_n$ and let $E(\overline{G})$ be the set of edges of $\overline{G}$.
Consider the vertices in the order $1, 2, \ldots, n$ and let $x_v$ 
be the `back-degree' of vertex $v$ in $\overline{G}$ (that is, the number of edges $uv$ in $\overline{G}$ with $u<v$). Then, by considering building up the linear extension step-by-step and using 
Observation~\ref{obs.nearchain}, we may see that 
\[ e(Q_n) \leq \prod_{v=1}^{n} (x_v+1) \leq  \left(\frac{1}{n} \sum_{v=1}^{n} (x_v+1) \right)^n, \]
by the arithmetic-geometric means inequality.  But  
\[ \sum_{v=1}^{n} (x_v+1) = |E(\overline{G})| +n \leq (1-\delta) \binom{n}{2}+n = n\left(\frac{1-\delta}{2} (n-1) +1\right).\]
Hence 
\[ e(Q_n) \leq \left(\frac{1-\delta}{2} n +1\right)^n = \left(\frac{1-\delta}{2} n\right)^n \left(1+ \frac2{(1-\delta)n}\right)^n.\]
But the last factor is at most $e^{\frac{2}{1-\delta}}$, and using $n^n \leq e^n n!$ concludes the proof. 
\end{proof}

It follows from the last lemma that if $\delta \geq 0$ and $1-\delta  \geq d/n$ for some $d>0$ then
\begin{equation} \label{eqn.fplusleq}
(f^+(n,\delta)/n!)^{1/n} \leq \tfrac12 e^{\frac{2}{d}+1} (1-\delta).
\end{equation}
Given $0 \leq \delta <1$ we may define $c_2(\delta)$ as follows.
Let $n_0=n_0(\delta)= 6/(1-\delta)$; and let
\begin{equation} \label{eqn.c2def}
c_2(\delta) = \min \left \{ 2^{-c \delta^2} \, , \:  2 (1-\delta) \right \},
\end{equation}               
where $c>0$ is the constant from Lemma~\ref{lem.fplus-ub1}.  
Then $c_2(\delta)<1$, $c_2(\delta) \to 0$ as $\delta \to 1$; and 
by Lemma~\ref{lem.maxub}, and inequality~\eqref{eqn.fplusleq} with $d=6$ (noting that $e^{4/3}<4$),
\[ (f^+(n,\delta)/n!)^{1/n} \leq c_2(\delta) \;\; \mbox{ for all } n \geq n_0,\]
as required in part (a) of Theorem~\ref{thm.main}.

Now we consider the very dense case, when we may determine $f^+(n,\delta)$ exactly. 
For, if $Q_n$ has $i$ incomparable pairs, then clearly $e(Q_n) \leq  2^{i}$.  Thus we always have
\begin{equation} \label{eqn.vdense0}
f^+(n,\delta) \leq 2^{\lfloor (1-\delta) \binom{n}{2} \rfloor}.
\end{equation}
But, if $i \leq n/2$, then $\bA(n,n-i)$ has $i$ incomparable pairs and $e(\bA(n,n-i))=2^i$.  Hence
\begin{equation} \label{eqn.vdense}
f^+(n,\delta) = 2^{\lfloor (1-\delta) \binom{n}{2} \rfloor} \;\;\;\; \mbox{ when } \;\; (1-\delta) \binom{n}{2} \leq n/2.
\end{equation}

Next we determine the rate of convergence of 
$(f^+(n,\delta)/n!)^{1/n}$ 
to $0$ as $\delta = \delta(n) \to 1$. 
\begin{prop} \label{prop.plusrate}
Let $\delta=\delta(n) \to 1$ as $n \to \infty$.  Then
\[ (f^+(n,\delta)/n!)^{1/n} = \Theta(\max \{1\!-\!\delta, 1/n\}). \]
\end{prop}
\begin{proof}
It suffices to consider the cases (a) $1-\delta \geq 1/n$ and (b)  $1-\delta \leq 1/n$.
In case (a), the lower bound from~\eqref{eqn.fplusgeq1} and the upper bound from~\eqref{eqn.fplusleq} show that
$(f^+(n,\delta)/n!)^{1/n} = \Theta(1\!-\!\delta)$.
In case (b),
$(1-\delta) \binom{n}{2} \leq n/2$, so by~\eqref{eqn.vdense0} we have $f^+(n,\delta) \leq 2^{n/2}$: thus
$1 \leq (f^+(n,\delta))^{1/n} \leq \sqrt{2}$, and so
\[ 1/n \leq (f^+(n,\delta)/n!)^{1/n} \leq e \sqrt{2}/n.\]   
This completes the proof.
\end{proof}

We proved two upper bounds on $e(Q_n)$ above, in Lemmas~\ref{lem.fplus-ub1} 
and~\ref{lem.maxub}. To close this section, we consider briefly whether 
the upper bound in~\eqref{eqn.KK} based on entropy, namely $e(Q_n) \leq 2^{nH(\overline{G})}$, could have been useful here.

Let $k$ be an integer at least $2$, let $\delta=\tfrac1{k}$, and let $Q_n$ be the balanced antichain example $\bA(n,k)$, with comparability graph $G$. We have $\comp(Q_n) \geq \delta \binom{n}{2}$ by~\eqref{eqn.antic2}.  
But $H(G) \leq \log_2 \chi(G) \leq \log_2 k$, so $H(\overline{G}) = \log_2 n - H(G) \geq \log_2(n/k)$.  
Hence the upper bound from~\eqref{eqn.KK} is
\[ 2^{nH(\overline{G})} \geq (n/k)^n. \]
In the case $\delta=\tfrac12$, the upper bound is at least $(n/2)^n \gg n!$,
so of course this tells us nothing.  
For smaller $\delta$ (when $k \geq 3$) we see that the upper bound from~\eqref{eqn.KK} is at least           
\[(n/k)^n = n! \, ((1+o(1))\, e \, \delta)^n, \]   
which is at least a factor of about $2^n$ greater than the upper bound in Lemma~\ref{lem.maxub}.

\section{Minimum numbers of extensions: $f^-(n,\delta)$}
\label{sec.min}

In this section we prove part (b) of Theorem~\ref{thm.main}, and give 
values for $c_3(\delta)$ and $c_4(\delta)$, see~\eqref{eqn.c3def} and~\eqref{eqn.c4def}.  
Also, in~\eqref{eqn.vsparse2} we give a formula for $f^-(n,\delta)$ in the 
very sparse case, and in Proposition~\ref{prop.minusrate}, we describe 
the rate of growth of $f^-(n,\delta)^{1/n}$ as $\delta=\delta(n) \to 0$.

\subsection*{Upper bounds}

Let $0< \delta \leq 1$.  Set $k=\lceil \delta^{-1} \rceil$,  
and let $Q_n$ be the balanced chain example $\bC(n,k)$. We have that $\comp(Q_n) \leq \frac{1}{k} \binom{n}{2} 
\leq \delta \binom{n}{2}$ by~\eqref{eqn.chaincomp}, and $e(Q_n) \leq k^n$ by~\eqref{eqn.chain1a}. Thus
\begin{equation} \label{eqn.fminusub1}
f^-(n,\delta)^{1/n} \leq \lceil \delta^{-1} \rceil \leq 2 \delta^{-1}.
\end{equation}
We wish also to show that $f^-(n,\delta)^{1/n} \to 1$ when $\delta = \delta(n) \to 1$ (which the last result does not give).  
Assume that $\tfrac12 \leq \delta < 1$. 
Let $t= \lceil (1-\delta) n \rceil$, let $\tau=t/n$ and note that $1-\delta \leq \tau < 1-\delta+1/n$.
Let $Q_n$ be the two-chain example $C(t, n-t)$.  By~\eqref{eqn.chain4} (since $t \leq (n+1)/2$\,) we have $\, \comp(Q_n) \leq (1-\tau) \binom{n}{2} \leq \delta  \binom{n}{2}$.
Hence by~\eqref{eqn.chain3}
\[ f^-(n,\delta)^{1/n} \leq e(Q_n)^{1/n} \leq \left(\frac{e}{\tau} \right)^{\tau} < \left(\frac{e}{1-\delta + 1/n} \right)^{1 - \delta +1/n}, \]
since $g(x)=(e/x)^x$ is increasing on $(0,1)$.  Thus if $\tfrac12 \leq \delta < 1$ and $1-\delta \geq 1/n$ then
\[ f^-(n,\delta)^{1/n} < \left(\frac{e}{2(1-\delta)} \right)^{2(1 - \delta)},\]
again using the fact that $g(x)$ is increasing. We now see that we can set
\begin{equation} \label{eqn.c4def}
c_4(\delta) = 
\begin{cases}
 \lceil \delta^{-1} \rceil 
   & \text{ if \,\,\,} 0<\delta < \frac12 \cr
 \min \left \{ \lceil \delta^{-1} \rceil \, , \: \left(\frac{e}{2(1-\delta)} \right)^{2(1 - \delta)} \right \} 
 & \text{ if \,\,\,} \frac12 \leq \delta <1.
\end{cases}                                  
\end{equation}
For, given $0 < \delta<1$, by~\eqref{eqn.fminusub1} and the last inequality we have
\[ f^-(n,\delta)^{1/n} \leq c_4(\delta) \;\; \mbox{ once }  n \geq 1/(1-\delta).\]
Note that the upper bound here tends to $1$ as $\delta \to 1$.

\subsection*{Lower bounds}

Let us state a preliminary inequality corresponding to inequality~(\ref{eqn.chain2b}).
For any poset $Q_n$ with height at most $k$ we have
\begin{equation} \label{eqn.height-elb}
e(Q_n)  \geq 2^{n-k}.
\end{equation}
To see this, let $Q_n$ have $h \leq k$ levels, and let level $i$ contain $r_i \geq 1$ elements for $i=1, \ldots, h$. Then, 
using the inequality $r_i! \geq 2^{r_i-1}$, we have 
\[ e(Q_n) \geq \prod_{i=1}^{h} r_i! \geq \prod_{i=1}^{h} 2^{r_{i}-1}= 2^{n-h} \geq 2^{n-k}. \]

Our first lower bound covers the whole range $0< \delta<1$, and in particular when $\delta$ is near to $1$.    
\begin{lem} 
\label{lem.half}
Let $0< \delta < 1$, 
and let $Q_n$ be a poset on $n$ vertices with $\comp(Q_n) \leq \delta \binom{n}{2}$.  Then
\[
e(Q_n) \geq  \max\{ 2^{(1-\sqrt{\delta})(n-1)}, 2^{\frac12(1-\delta)n} \}. 
\]
\end{lem}
\noindent
Note that $1-\sqrt{\delta} > \frac12(1-\delta)$, and $1-\sqrt{\delta} = \frac12(1-\delta) +O((1-\delta)^2)$ as $\delta \to 1$.   Thus the first term in the maximum above is slightly better in terms of $\delta$.  The second term is the one we use to define $c_3$ in~(\ref{eqn.c3def}) below, where we need the factor $n$ not $(n-1)$ in the exponent.
\begin{proof}
For both of the terms in the maximum, we will use inequality~\eqref{eqn.height-elb}, together with an upper bound on the height $h=h(Q_n)$ following from $\binom{h}{2} \leq \delta \binom{n}{2}$.
Note first that $h<\sqrt{\delta}(n-1)+1$; for otherwise
\[  \binom{h}{2} \geq \tfrac12 (\sqrt{\delta}(n-1)+1) \sqrt{\delta}(n-1)
> \tfrac12 (\sqrt{\delta} n) \sqrt{\delta}(n-1) = \delta \binom{n}{2}.
\]
Thus by~(\ref{eqn.height-elb}) we have $e(Q_n) >  2^{(1-\sqrt{\delta})(n-1)}$.

It remains to show that
$ e(Q_n) \geq 2^{\frac12(1-\delta)n}$.
Since $\delta < 1$ we have $e(Q_n) \geq 2$, so we may assume that $(1 - \delta)n > 2$.
We {\bf claim} that $h< (1+\delta)n/2$: once we have proved this, the required lower bound on $e(Q_n)$ will follow directly from~(\ref{eqn.height-elb}). 
But if $h \geq (1+\delta)n/2$, then 
$\delta \binom{n}{2} \geq \binom{h}{2} \geq \frac{1}{8} (1+\delta)n \, ((1+\delta)n-2)$.  
This may be rewritten as
\[ n \delta^2 -2(n-1)\delta +n-2 \leq 0;\]
and this quadratic inequality gives $(1-\delta) n \leq 2$, which contradicts $(1-\delta)n > 2$. 
This completes the proof of the claim, and thus of the lemma. 
\end{proof}

We may see quickly that $f^-(n, \delta)^{1/n} \to \infty$ when $n \to \infty$ and $\delta \to 0$, as required   
in part (b) of Theorem~\ref{thm.main}.  
We noted in the proof of the last lemma that the height $h=h(Q_n)$ satisfies
$h<\sqrt{\delta}(n-1)+1$, so $h<\sqrt{\delta}n+1$.
Hence, by inequality~\eqref{eqn.chain2b} (and using $n! \geq (n/e)^n$), 
\[ e(Q_n)^{1/n} \geq \frac{n}{e(\delta^{\frac12}n+1)} =
\frac1{e}  \, \delta^{-\frac12} \left (1+\tfrac{1}{\delta^{\frac12}n} \right)^{-1}.\]
But, if $n \geq 11 \delta^{-\frac12}$, then
\[ \left(1+\tfrac{1}{\delta^{\frac12}n}\right)^{-1} \geq 1-\tfrac{1}{\delta^{\frac12}n} \geq 
\tfrac{10}{11} > \tfrac{e}{3}.\]  
Hence, for each $0<\delta \leq 1$,
\[ f^-(n,\delta)^{1/n} 
\geq \tfrac1{3}  \delta^{-\frac12} \;\;\; \mbox{ for all } n \geq 11 \delta^{-\frac12}.\]
This gives a lower bound as required on $f^-(n,\delta)^{1/n}$; 
but we can obtain better lower bounds for this case
in terms of the dependence on $\delta$. In the bound in Lemma~\ref{lem.low} 
below we essentially replace $\delta^{-\frac12}$ by $\delta^{-1}$.

We first recall a preparatory lemma.  A proof of this can be found in for example~\cite{bw} 
Theorem 6.3 -- the result is also stated as exercise 57 to Chapter 3 of~\cite{stan} 
and the equality part to exercise 20 in~\cite[subsection 5.1.4]{Knu}.  

\begin{lem}
\label{lem.prel}
Let $Q_{n}$ be a poset on $n$ vertices, and for each $t \in Q_n$ let 
$\lambda_t:= \vert \{s \in Q_n: s \preceq t \}\vert$ be the size of the 
principal downset for $t$ (often referred to as the hook length). Then 
\[
e(Q_n) \geq \dfrac{n!}{\prod_{t \in Q_n} \lambda_{t}},
\]
with equality if and only if $Q_n$ is a downward-branching forest. 
\end{lem}

\begin{lem}
\label{lem.low}
Let $0 < \delta \leq 1$ and let $Q_n$ be a poset on $n$ vertices with 
$\comp(Q_n) \leq \delta \binom{n}{2}$. Then 
\[
e(Q_n) \geq n! \left(\dfrac{2}{\delta(n-1)+2}\right)^{n} \geq  
e^{1-2/\delta} \left(\dfrac{2}{e\delta}\right)^{n}.
\]
\end{lem}
\begin{proof}
With notation as in the last lemma,
\begin{equation*}
\sum_{t \in Q_n} \lambda_{t}=\comp(Q_n)+n,
\end{equation*} 
and by the arithmetic-geometric means inequality 
\[
\left(\frac{\sum_{t \in Q_n} \lambda_{t}}{n} \right)^{n} \geq \prod_{t \in Q_n} \lambda_{t}.
\]
Hence, by Lemma~\ref{lem.prel}
\begin{eqnarray*}
e(Q_n) & \geq & \dfrac{n!}{\prod_{t \in Q_n} \lambda_{t}} 
\;\; \geq \;\; n!\left(\dfrac{n}{\sum_{t \in Q_n} \lambda_{t}}\right)^{n} \\
& = & n! \left(\dfrac{n}{n+\comp(Q_n)}\right)^{n} 
\;\: \geq \;\; n! \left(\dfrac{n}{n+ \delta \binom{n}{2}}\right)^{n} \\
& = & n! \left(\dfrac{2}{\delta(n-1)+2}\right)^{n},
\end{eqnarray*} 
which gives the first inequality required.
Using $n! \geq \left(\frac{n}{e} \right)^{n}$ and $1+x \leq e^{x}$, we see that the last term is at least
\begin{eqnarray}
\label{eqn.simplified}
\left(\dfrac{n}{e}\right)^{n} \left(\dfrac{2}{\delta (n-1)+2}\right)^{n} \nonumber 
& = & \left(\dfrac{2}{e\delta}\right)^{n} \left(1+\dfrac{\frac{2}{\delta}-1}{n} \right)^{-n} \nonumber 
\; \geq \; \left(\dfrac{2}{e\delta}\right)^{n} e^{1-2/\delta},
\end{eqnarray}
which completes the proof.
\end{proof}
By the last lemma, if $n \geq 6/\delta$ then
\[ f^-(n,\delta)^{1/n} \geq e^{-\tfrac{2}{\delta n}} \dfrac{2}{e\delta} 
\geq e^{-\tfrac13} \dfrac{2}{e\delta}  \geq \dfrac{1}{2\delta}.\]
Thus, using also Lemma~\ref{lem.half}, we see that we may set 
\begin{equation}\label{eqn.c3def}
c_3(\delta) = \max \left \{ 2^{\frac12(1-\delta)} \, , \: \frac{1}{2\delta} \right \}.
\end{equation}

Next we consider the very sparse case. By Lemma~\ref{lem.low}, if $\comp(Q_n)=j$ 
(so the density $\delta=\delta(Q_n)$ satisfies $\delta(n\!-\!1)/2= j/n$) then  
\begin{equation} \label{eqn.vsparse}
e(Q_n) \geq n! \, (1+\delta(n\!-\!1)/2)^{-n} = n! \, (1+j/n)^{-n} \geq n! \, e^{-j}.
\end{equation}
If $\comp(Q_n)=j \leq n/2$, then we can have $e(Q_n)= n! \, 2^{-j}$ (when the comparability graph is a matching).  
Inequality~\eqref{eqn.vsparse} is sufficient to prove Proposition~\ref{prop.minusrate} 
below, but we can improve it, and obtain the equality~\eqref{eqn.vsparse2} which matches~\eqref{eqn.vdense} nicely.  
We need a preliminary result, perhaps of interest in its own right. 
\begin{prop} \label{lem.vsparse}
For each partial order $Q$ on $[n]$ we have
\[ e(Q) \geq n! \, 2^{-\comp(Q)}, \]
and equality holds if and only if the comparability graph is a (partial) matching.  
\end{prop}

\begin{proof} 
Let us think of a partial order $Q$ as the set of ordered pairs $(u,v)$ of distinct elements such that $u \prec v$ in $Q$.  Observe that $\comp(Q)=|Q|$.
Suppose that $|Q|=k$. There is an increasing sequence $Q_0, Q_1, \ldots, Q_k$ of partial orders on $[n]$ such that $|Q_i|=i$ for $i=0, \ldots, k$ and $Q_k=Q$. 
Thus the inequality in the proposition will follow from the next claim.  

\begin{cl-non-num}
Let $Q$ and $Q'$ be partial orders on $[n]$ such that $Q'=Q \cup \{(u,v)\}$ where $(u,v) \not\in Q$.
Then $e(Q) \leq 2 e(Q')$.
\end{cl-non-num}

To prove the claim, note first that $(x, u) \in Q'$ implies $(x, v) \in Q$ (since $(x, v) \in Q'$ by transitivity, and $(x,v) \neq (u,v)$), and similarly $(v, y) \in Q'$ implies $(u, y) \in Q$. Let $LE(Q)$ denote the set of linear extensions of $Q$. 
Let $\pi \in LE(Q)$ with $\pi(u)>\pi(v)$, and let $\pi'$ be formed from $\pi$ 
by swapping $u$ and $v$: it suffices to show that $\pi' \in LE(Q')$.  
Consider $u$ first.
If $(x, u) \in Q'$ then, since $(x, v) \in Q$ (as we noted), we have $\pi'(x) = \pi(x) < \pi(v) = \pi'(u)$; $\pi'(u) < \pi'(v)$; and if $(u, y) \in Q'$ for some $y \neq v$  then $(u,y) \in Q$ and so $\pi'(u) < \pi(u) < \pi(y) = \pi'(y)$.
Considering $v$ similarly, we see that $(x, v) \in Q'$ implies $\pi'(x) < \pi'(v)$ and $(v, y) \in Q'$ implies $\pi'(v) < \pi'(y)$.
It follows that $\pi' \in LE(Q')$, as required.  

Finally, consider when equality holds.  By symmetry, we must have equality when the comparability graph is a matching.  
Now suppose it is not a matching, 
and note that we may choose $Q_2$ to consist of two comparable pairs $(u,x), (v,x)$ or $(x,u), (x,v)$.  
But now $e(Q_2)=n! /3$; and so, arguing as before,
\[ e(Q) \geq (4/3) \, n! \,  2^{-\comp(Q)} > n! \,  2^{-\comp(Q)}, \]
which completes the proof.
\end{proof}
By the last result and the discussion preceding it we have    
\begin{equation} \label{eqn.vsparse2}
f^-(n,\delta) = n! \, 2^{-\lceil \delta \binom{n}{2}\rceil} \;\;\; \mbox{ when } \;\; \delta (n-1) \leq 1.
\end{equation}

Now we determine the rate of growth of $f^-(n,\delta)^{1/n}$ as $\delta=\delta(n) \to 0$. 
\begin{prop} \label{prop.minusrate}
Let $\delta=\delta(n) \to 0$ as $n \to \infty$.  Then
\[ f^-(n,\delta)^{1/n} = \Theta(\min \{\delta^{-1}, n\}). \]
\end{prop}
\begin{proof}
It suffices to consider the cases (a) $\delta \geq 1/n$ and (b)  $0<\delta \leq 1/n$.
In case (a), the upper bound from~\eqref{eqn.fminusub1} and the lower bound from Lemma~\ref{lem.low} show that
$f^-(n,\delta)^{1/n} = \Theta(\delta^{-1})$.
In case (b),  let $j= \lceil \delta \binom{n}{2} \rceil$, and note that $j \leq n/2$.
Hence if $\comp(Q_n) \leq j$ then by~\eqref{eqn.vsparse} or Proposition~\ref{lem.vsparse}
\[ e(Q_n) \geq n! \, e^{-j} \geq n! \, e^{-n/2}.\]
Thus
\[n \geq (n!)^{1/n} \geq f^-(n,\delta)^{1/n} \geq (n!)^{1/n} e^{-\tfrac12} \geq n e^{-\tfrac32},\]
and so $f^-(n,\delta)^{1/n}=\Theta(n)$, as required.
\end{proof}

\subsection*{Example: the lattice of subsets of $[t]$}

An interesting example of a poset is the lattice of $n=2^t$ subsets of $[t]$, 
ordered by inclusion. Let us call this poset $L_n$. 
It is easy to see that there are $3^t$ pairs $A,B$ with $A \subseteq B \subseteq [t]$, 
and it follows that $\comp(L_n)= 3^t-2^t=n^2 (\tfrac34)^t -n$: thus  
\[\delta(n):= \frac{\comp(L_n)}{\binom{n}{2}} \sim 2 \cdot \left(\frac34 \right)^{\log_2 n} = 2 n^{-\log_2 (4/3)} = o(1). \]
(Note that $\log_2(4/3) \approx 0.415$.)
Brightwell and Tetali~\cite{Brigh}, improving on~\cite{SK}, give a very 
good estimate of $e(L_n)$, which by Stirling's formula
implies that
\[ ( e(L_n)/n!)^{1/n} \sim ((\pi e/2) \log_2 n)^{-1/2}. \]
Thus we are not close to having as many linear extensions as possible 
(given $\delta=\delta(n)$), but closer than to the opposite case of having as few as possible. 
Indeed, as $n \to \infty$, by part (a) of Theorem~\ref{thm.main}, $( f^+(n,\delta)/n!)^{1/n} \to 1$, 
and by Proposition~\ref{prop.minusrate}, $f^-(n,\delta)^{1/n} = \Theta( \delta^{-1}) =\Theta ( n^{\log_2(4/3)})$.

\medskip

Let us close this section by considering briefly whether the lower bounds~\eqref{eqn.KK} 
and~\eqref{eqn.Cardinal} on $e(Q_n)$ based on entropy can yield good lower bounds on $f^-(n,\delta)$. 
The first lower bound $n! \, 2^{-n H({G})}$ tells us nothing when $\delta$ is at least about $\frac12$.
For let $Q_n$ be the balanced chain example $\bC(n,2)$, with comparability graph $G$. 
Then $\comp(Q_n) \sim  \frac{1}{2} \binom{n}{2}$. Also $\chi(\overline{G})=2$, 
so $H(\overline{G}) \leq 1$ and thus $H(G) \geq \log_{2} (n/2)$. Hence
\[ n! \, 2^{-n H(G)}  \leq   n! \, 2^{-n \log_2 (n/2)} \; = n! \, (\tfrac{2}{n})^n  = o(1).\]

Now consider smaller $\delta$, say $\delta=1/k$ where $k \geq 3$.
Consider the balanced chain example $\bC(n,k)$. 
Arguing as above we have $\chi(\overline{G})=k$, so $H(G) \geq \log_2(n/k)$.  Hence
\[ n! \, 2^{-n H(G)}  \leq   n! \, 2^{-n \log_2 (n/k)} \; = n! \, (\tfrac{k}{n})^n  = \left( \tfrac{1+o(1)}{e \delta} \right)^n. \]
Thus we do not quite match the bound $(\tfrac{1}{2\delta})^n$ above. 

The second lower bound~\eqref{eqn.Cardinal}, namely $2^{\frac{1}{2} n H(\overline{G})}$, does give a lower bound of the form $2^{\Omega(n)}$ for each $0<\delta<1$.  For if $\comp(Q_n) \geq \delta \binom{n}{2}$, 
then as we have seen $e(Q_n) \geq c_3 (\delta)^n$ for $n$ sufficiently large, and so 
for the comparability graph $G$ we have by the upper bound in~\eqref{eqn.KK} that
\[ 2^{\frac{1}{2} n H(\overline{G})} \geq e(Q_n)^{\frac12} \geq (c_3(\delta)^{\frac12})^n. \]
In particular this gives
\[ f^-(n,\delta)^{1/n} \geq 2^{\frac{1}{2} H(\overline{G})} = \Omega(\delta^{-\frac12}) \;\; \mbox{ as } \delta \to 0. \]

However, using~\eqref{eqn.Cardinal} we do not obtain $f^-(n,\delta)^{1/n} = \Omega(\delta^{-1})$ as $\delta \to 0$.
For, let $k=\lceil \delta^{-1} \rceil$, and let $Q_n$ be the balanced chain 
example $\bC(n,k)$ (as at the start of this section), with comparability graph $G$.  
Recall that $\comp(Q_n) \leq \delta \binom{n}{2}$, and 
$H(\overline{G}) \leq \log_{2} k$. 
Thus, 
\[ 2^{\frac12 H(\overline{G})} \leq k^{\frac12} = {\lceil \delta^{-1} \rceil}^{\frac12} < 2 \delta^{-\frac12}. \]

\section{Random interval partial orders}
\label{sec.randint}

Here we examine random interval posets, see~\cite{fishburn}.
Given a family of $n$ closed intervals $I_j=[a_j,b_j]$ in the real line for $j = 1, \ldots, n$, 
we form an \emph{interval partial order} on $\{1, \ldots, n\}$ by setting $i \prec j$ if and only if $b_i < a_j$. 
Thus two elements of the poset are comparable when their corresponding intervals 
do not intersect, and otherwise they are incomparable. A chain of the poset 
is a set of pairwise non-intersecting intervals; and an antichain is a 
set of pairwise intersecting intervals (that is, by Helly's Theorem, a set of intervals containing a common point).

Suppose that we start with $2n$ independent random variables $X_{1},\ldots, X_{n}$ and 
$Y_{1},\ldots, Y_{n}$, each from the uniform distribution on $[0,1]$; and form $n$ closed 
intervals $I_{j}$, for $j = 1,\ldots,n$, where $I_{j}=[X_{j}, Y_{j}]$ 
if $X_{j}<Y_{j}$ and $[Y_{j},X_{j}]$ otherwise. 
(The event that $X_{j}=Y_{j}$
has probability zero so can be ignored.) These intervals $I_j$ yield a 
\emph{random interval poset} $P_n$. It would not matter if we 
replaced the uniform distribution on $[0,1]$ by any continuous distribution.  
Also equivalently, we could generate the intervals from a 
uniform random perfect matching on $\{1, \ldots, 2n\}$.

It is easy to see that the probability that two random intervals are 
comparable is $1/3$ -- see e.g. \cite{jsw}.
Indeed, for $i \neq j$, the probability that the four end points of 
the intervals $I_i$ and $I_j$ come in any given order is $1/ 4! = 1/24$; 
and there are exactly $8$ orders such that the intervals are disjoint. 
It follows that the number $Z$ of edges in the comparability graph 
satisfies $\mathbb{E}[Z]=\frac{1}{3} \binom{n}{2}$.

Also, given any values $x_1, \ldots, x_{n}$, $y_1, \ldots, y_{n}$, 
as end points of $n$ intervals, changing the values of $x_k$ and $y_k$
can affect only the edges incident to node $k$ in the comparability graph, 
so the number of edges can vary by at most $n-1$. 
Hence by the bounded differences inequality (Lemma 1.2 of~\cite{McD}, 
or see for example Theorem 6.3 of~\cite{blm}), for each $t>0$,  
\[ 
\mathbb{P}(\vert Z - \mathbb{E}[Z] \vert \geq t) \leq 2 \exp \left (-\frac{2t^2}{n(n-1)^2} \right).
\]
Thus, for $0 < a < 1$
\[ 
\mathbb{P} \left (\bigg \vert Z- \frac{1}{3} \binom{n}{2} \bigg \vert 
\geq a \binom{n}{2} \right) \leq 2 \exp \bigl( - \frac{1}{2} a^2 n \bigr)
\]
and we see that a random interval order is dense {\whp}.

Again, we want to estimate how many linear extensions $P_n$ has.
Let $A_{n}$ denote the size of a largest set of intersecting intervals in our 
family of $n$ random intervals -- in the poset language, this is a maximum antichain.
It was shown in \cite{jsw} that {\whp}
$A_{n} = \frac{n}{2}+o(n)$.
But $\: e(P_n) \geq A_n!\;$ 
so by Stirling's formula, we have that {\whp}
\[
\log_2 e(P_n) \geq \log_2(\frac{n}{2}+o(n))! = (\frac{1}{2} + o(1)) n \log_2 n. 
\]
We shall refine the idea above to show that a random interval partial order has many more linear extensions.
\begin{thm}
\label{thm.randint}
The number $e(P_n)$ of linear extensions 
of a random interval partial order $P_n$ satisfies
\begin{equation*}
e(P_n) = n! \, 2^{-\Theta(n)} \;\; \whp.
\end{equation*}
\end{thm}
\begin{proof}
By Lemma~\ref{lem.fplus-ub1} it will suffice for us to show that $e(P_n) \geq n! \, 2^{-O(n)} \, \whp$; 
and in fact we shall show that
\begin{equation} 
\label{eqn.randint}
e(P_n) \geq n! \, 2^{-(6+o(1))n} \;\; \whp.
\end{equation}
For positive integers $i$ and $j$ with $i$ odd and $i<2^j$, let $I(i, j)$ 
denote the interval $(\frac{i-1}{2^{j}}, \frac{i+1}{2^{j}})$.  
Thus for each $j \geq 1$ there are 
$2^{j-1}$ intervals $I(i, j)$; and 
these intervals partition $[0,1]$, except that they omit the $2^{j-1}+1$ 
end points $\frac{i}{2^{j-1}}$ for 
$0 \leq i \leq 2^{j-1}$.

Let $0 < a < b < 1$ and consider the interval $(a, b)$.
Let $j(a, b)$ be the least integer $j \geq 1$ such that $\frac{i}{2^{j}} \in (a, b)$
for some positive integer $i$. There is a unique such $i$: for if there were
at least two such $i$, then there would be at least one even $i$,
and we could replace $i/2^{j}$ by $(i/2)/2^{j-1}$, giving a smaller value of $j$ 
and contradicting the definition of $j$. Since $i$ is unique, we may call it $i(a, b)$.

We denote the interval $I\bigl(i(a, b), j(a, b)\bigr)$ by $J(a, b)$.
For positive integers $i$ and $j$ with $i$ odd and $i<2^j$,
let $\mathcal{A}(i, j)$ be the random set of all intervals $(a,b)$ amongst $I_1, \ldots, I_n$ 
such that $J(a, b)=I(i, j)$.
Each interval $(a, b)$ in $\mathcal{A}(i, j)$ contains the point $i/2^j$, so
the sets $\mathcal{A}(i, j)$ are antichains.
It follows also that, if the midpoint of $I(i, j)$ (i.e. $i/2^{j})$ is
less than the midpoint of $I(i^{\prime}, j^{\prime})$ (i.e. $i^{\prime}/2^{j^{\prime}}$)
then no interval in $\mathcal{A}(i^{\prime}, j^{\prime})$ can precede
any interval in $\mathcal{A}(i, j)$ in the interval partial order.

Let $N(i,j) = \vert \mathcal{A}(i, j) \vert$, so $\sum_{i,j}N(i,j)=n$.  Then from the above
\[ e(P_n) \geq \prod_{i,j} N(i,j)! \geq \prod_{i,j} (N(i,j)/e)^{N(i,j)} \] 
and so
\[ \log_2 e(P_n) \geq 
\sum_{i, j} N(i,j) \log_2 N(i,j) - n \log_2 e. \]
Here we sum only over $i,j$ such that $N(i,j) \neq 0$.
Fix~$j$, and let $N_j = \sum_i N(i,j)$. Then, since we sum over at most 
$2^{j-1}$ values of $i$, by convexity of $x \log_2 x$
\[ \sum_i  N(i,j) \log_2 N(i,j)  \geq N_j \log_2 (N_j / 2^{j-1}). \]
Hence
\begin{equation} \label{eqn. logeP}
\log_2 e(P_n) \geq 
\sum_{j} N_j \log_2 (N_j / 2^{j-1}) - n \log_2 e.
\end{equation}

Let the random variable $J_k$ be $j(a,b)$ for the $k$th random interval.
Then ${\mathbb P}(J_k=j) = 2^{-j}$ for each $j=1,2,\ldots$.
To see this, condition on the value $x$ of the first end point of the interval 
(which could be the left or the right end point).  If $x \in (\frac{i-1}{2^j}, \frac{i}{2^j})$ 
for some odd $i$, then $J_k=j$ if and only if the second end point falls in
$(\frac{i}{2^j}, \frac{i+1}{2^j})$, which has probability $2^{-j}$.
On the other hand, if $x \in (\frac{i}{2^j}, \frac{i+1}{2^j})$ for some odd $i$, 
then $J_k=j$ if and only if the second end point falls in
$(\frac{i-1}{2^j}, \frac{i}{2^j})$, which also has probability $2^{-j}$.

By the definition of $N_j$, it is the number of intervals $I_k$ such that $J_k=j$.
Then $N_j$ has the binomial distribution with parameters $n$ and $2^{-j}$, 
which has mean ${\mathbb E}(N_j)= n 2^{-j}$ and variance $n 2^{-j}(1-2^{-j}) \leq n 2^{-j}$. 
Hence, by Chebyshev's inequality,
\[ {\mathbb P} ( |N_j - n 2^{-j}| \geq n 2^{-j-1}) 
\leq \frac{n 2^{-j}}{ (n 2^{-j-1})^2} = \frac{4}{n} \, 2^{j}.\]
Let $j_0 := j_0(n) = \log_2 n - h(n)$, 
where $h(n) \to \infty$ slowly as $n \to \infty$: 
we set $h(n) \sim \log_2\log_2 n$. Then, by the union bound,
\begin{eqnarray*}
&&
{ \mathbb P} ( |N_j - n 2^{-j}| \geq n 2^{-j-1} \mbox{ for some } j=1,\ldots, j_0 ) \\
& \leq &
\frac{4}{n} \sum_{j=1}^{j_0} 2^j \; \leq \; \frac{8}{n}\, 2^{j_0} \; = \; 2^{-h(n)+3} \; = \; o(1).
\end{eqnarray*}
Also,
\[ { \mathbb E}( \sum_{j>j_0} N_j ) = n \sum_{j>j_0} 2^{-j} \leq n 2^{-j_0} = 2^{h(n)}.\]
Hence by Markov's inequality, {\whp}
\[ \sum_{j>j_0} N_j \leq h(n) 2^{h(n)} \;\; \mbox{ and so } \;\;
\sum_{j=1}^{j_0} N_j \geq n - h(n) 2^{h(n)}.\]

Let $B_n$ be the event that $|N_j - n 2^{-j}| \leq n 2^{-j-1} \mbox{ for each } j=1,\ldots,j_0$
and $\sum_{j=1}^{j_0} N_j \geq n - h(n) 2^{h(n)}$.  
By the above $ {\mathbb P}(B_n) =1-o(1)$. Condition on $B_n$.  
Then by~\eqref{eqn. logeP}
\begin{eqnarray*}
\log_2 e(P_n) + n \log_2 e
& \geq &
\sum_{j=1}^{j_0} N_j \log_2 (N_j 2^{-j+1})\\
& \geq &
\sum_{j=1}^{j_0} N_j \log_2 (n 2^{-j-1} 2^{-j+1}) \\
& = &
\sum_{j=1}^{j_0} N_j \log_2 n - \sum_{j=1}^{j_0} N_j \cdot 2j.
\end{eqnarray*}
But (still conditioned on the event $B_n$)
\[ 
\sum_{j=1}^{j_0} N_j \cdot 2j \leq (3/2) n \sum_{j=1}^{j_0} 2^{-j} \cdot 2j \leq c n,
\]
where $c= (3/2) \sum_{j \geq 1} 2^{-j} \cdot 2j=6$. Hence
\begin{eqnarray*}
\log_2 e(P_n) & \geq &
\left(\sum_{j=1}^{j_0} N_j \right) \log_2 n  - (6 + \log_2 e) n \\
& \geq &
(n - h(n) 2^{h(n)} ) \log_2 n - (6 + \log_2 e) n \\
& = & n \log_2 (n/e) - (6+o(1)) n.
\end{eqnarray*}
We have now shown that the event $B_n$ happens {\whp}, and when it happens 
we have $e(P_n) \geq n! \, 2^{-(6+o(1))n}$.  Hence~\eqref{eqn.randint} holds, and we are done.
\end{proof}

From Theorem~\ref{thm.randint}, together with the upper bound in part (a) of 
Theorem~\ref{thm.main}, we see that the random interval order has about `as many linear extensions as possible'.  
Another model of random partial order is the random $k$-dimensional 
partial order $R_n$.  To form this we take $k$ independent random linear 
orders $\leq_i$ each uniformly distributed on $[n]$, and set  
$x \preceq y$ in $R_n$ when $x \leq_i y$ for each $i=1,\ldots,k$.  
Brightwell~\cite{brightwell} has given a precise estimate for $e(R_n)$, which shows that 
$ e(R_n)^{1/n} = \Theta(n^{1-\tfrac1{k}})$  {\whp}; 
thus we are meeting numbers of linear extensions spread more widely 
through the range of possible values given in Theorem~\ref{thm.main}.  

It would be interesting to compute estimates on the number of linear 
extensions of other, random or deterministic,
models of partial orders; see for example~\cite{geo} and~\cite{trotter} for two such models.

\section{Concluding remarks}
\label{sec.concl}

We have proved Theorem~\ref{thm.main}, concerning maximum and minimum 
possible numbers of linear extensions, and in the process we have given more detailed information. 
In particular, we saw that if $\delta \to 1$ but not too quickly then
$(f^+(n,\delta)/n!)^{1/n} = \Theta(1\!-\!\delta)$; and if $\delta \to 0$ 
but not too quickly then $f^-(n,\delta)^{1/n} = \Theta(\delta^{-1})$.
What can we say about other rates of convergence? Can we improve the values $c_i(\delta)$? 

Let us focus on $\delta$ of the form $1/k$ for an integer $k \geq 2$.  
From our earlier results, $\comp(\bA(n,k)) \sim \delta \binom{n}{2}$ and 
$e(\bA(n,k)) = n! \, (1-\delta)^n \, 2^{O(\log_{2} n)}$.  Is $\bA(n,k)$ extremal for having many linear extensions? 
Similarly, $\comp(\bC(n,k)) \sim \delta \binom{n}{2}$ and $e(\bC(n,k)) = 
\delta^{-n} \, 2^{O(\log_{2} n)}$; and we may ask if $\bC(n,k)$ is extremal for having few linear extensions.

\begin{conj} \label{conj.plus}
Let $\delta =1/k$ for an integer $k \geq 2$. Then (a)
\[ (f^+(n,\delta)/n!)^{1/n}  \to 1- \delta \;\; \mbox{ as } n \to \infty;\]
and (b), more boldly, each partial order $Q_n$ on $[n]$ with 
$\comp(Q_n) \geq \comp(\bA(n,k))$ satisfies $e(Q_n) \leq e(\bA(n,k))$.
\end{conj}           

\begin{conj} \label{conj.minus}
Let $\delta =1/k$ for an integer $k \geq 2$. Then (a)
\[ f^-(n,\delta)^{1/n}  \to 1/\delta \;\; \mbox{ as } n \to \infty;\]
and (b), more boldly, each partial order $Q_n$ on $[n]$ with 
$\comp(Q_n) \leq \comp(\bC(n,k))$ satisfies $e(Q_n) \geq e(\bC(n,k))$.
\end{conj}
                       
For partial orders $Q_n$ on $n$ points with a given density of edges in 
the comparability graph $G$, we have bounded the possible number of 
linear extensions, which equals $n!$ times the volume of the clique polytope $\C(G)$. 
It could be interesting to investigate bounds on the volume of 
the clique polytope $\C(G)$ for graphs $G$ from other classes of perfect graphs (or indeed from more general classes).

We have also shown in Theorem~\ref{thm.randint} that a random 
interval partial order $P_n$ \whp\ has $e(P_n) = n! \, 2^{-\Theta(n)}$.  
Can this be pinned down more precisely? Is there a constant $c>0$ 
such that \whp\ $(e(P_n)/ n!)^{1/n} \to c$?       
(If there is such a constant $c$, then $c \geq 2^{-6}$ by the inequality~\eqref{eqn.randint}.)

\subsubsection*{Acknowledgements}

We are grateful to the referees, whose comments have led to a much improved paper, 
and have encouraged us for example to make explicit the Conjectures~\ref{conj.plus} and~\ref{conj.minus}.

\end{document}